\documentclass{amsart}
\usepackage[english]{babel}
\usepackage{amsmath,amssymb}
\usepackage{amsthm}
\usepackage{shuffle} 

\usepackage{xcolor}
\newcommand{\assign}{:=}
\def\RR{{\mathbb R}}
\newcommand{\BCH}{\operatorname{BCH}}
\newcommand{\tmop}[1]{\ensuremath{\operatorname{#1}}}

\newcommand{\tmtextit}[1]{\text{{\itshape{#1}}}}

\usepackage[hidelinks]{hyperref}
\usepackage{enumitem} 

\newtheorem{theorem}{Theorem}
\newtheorem{definition}[theorem]{Definition}
\newtheorem{corollary}[theorem]{Corollary}
 
\newtheorem{example}[theorem]{Example} 
\newtheorem{remark}[theorem]{Remark} 
 
\numberwithin{theorem}{section}

\title{Rectifiable paths with polynomial log-signature are straight lines}

\author{Peter K. Friz}
\address{Technische Universit\"at Berlin, WIAS Berlin,
  Strasse des 17. Juni 136, 10623 Berlin, Germany
}
\email{friz@math.tu-berlin.de}
\thanks{}

\author{Terry Lyons}
\address{University of Oxford,
Radcliffe Observatory, Andrew Wiles Building, Woodstock Rd, Oxford OX2 6GG, UK}
\email{tlyons@maths.ox.ac.uk}
\thanks{}

\author{Anna Seigal}
\address{Harvard University, 29 Oxford Street, Pierce Hall 324, Cambridge, MA 02138, USA}
\email{aseigal@seas.harvard.edu}
\thanks{}

\subjclass[2020]{60L10, 60L70, 60G10, 17B01, 15A69, 13P15}

\keywords{Signature, log-signature, Lie series}

\begin{document}

\maketitle

\begin{abstract}
The signature of a rectifiable path is a tensor
series in the tensor algebra whose coefficients are definite iterated integrals of the path. 
The signature characterises the path up to a generalised form of reparametrisation.
It is a
classical result of K. T. Chen
that the log-signature (the logarithm of the signature) is a Lie series. 
A Lie series is polynomial if it has finite degree.
We show that the log-signature is polynomial if and only if the path is a straight line up to reparametrisation.
Consequently, the log-signature of a rectifiable path either has degree one or infinite support. 
Though our result pertains to rectifiable paths, the proof uses results from rough path theory, in particular that
the signature characterises a rough path up to reparametrisation.
\end{abstract}

\section{Introduction and main result}

Classically,
a path in $\RR^d$ is a function $\gamma: [t_0,t_1] \to \RR^d$, with $\gamma(t)$ the value of the path at time~$t$.
We call a path 
{\em rectifiable} if it is continuous and of bounded variation.
The tensor algebra $T
((\mathbb{R}^d))$ 
is the space of tensor series over $\RR^d$.
It is isomorphic to the space of power series in~$d$
non-commuting indeterminates. 
The \emph{signature} of a path is a tensor series in the tensor algebra. 
The entries of the signature are features that encode the path.
The signature is the terminal solution (at time $t = t_1)$ to the controlled differential equation 
\begin{equation} 
\label{eqn:cde} d
\mathbf{S} = \mathbf{S} \otimes d \gamma, \qquad  \mathbf{S}(t_0) = 1 \in T
((\mathbb{R}^d)).
\end{equation} 
We denote the signature of the path $\gamma$ by $\tmop{Sig} (\gamma) := \mathbf{S}(t_1)$.
It is also called the {\em signature transform} of $\gamma$ (on $[t_0,t_1]$).
The level $N$
truncation of the tensor algebra is $T^{(N)} (\mathbb{R}^d):= \bigoplus_{k=0}^N(\RR^d)^{\otimes k}$.
Projecting $\tmop{Sig} (\gamma)$ to level~$k$ gives a tensor in $(\RR^d)^{\otimes k}$.
In coordinates, its coefficient at position $(i_1, \ldots, i_k)$ is the iterated integral
\begin{equation}
\label{eqn:first}
\int_{s_1 =0}^{t_1} \left( \int_{s_2=0}^{s_1} \cdots \left( \int_{s_k = t_0}^{s_{k-1}} d \gamma_{i_k} ( s_k) \right) \cdots d\gamma_{i_{2}}(s_{2}) \right) d\gamma_{i_1}(s_1) ,
\end{equation}
with all integrals understood in classical Riemann-Stieltjes sense. 
The signature can be defined on any sub-interval of $[t_0, t_1]$.
The {\em indefinite signature} of $\gamma$ 
is the path $\mathbf{S} : [t_0,t_1] \to T((\RR^d))$, which satisfies $\mathbf{S}(t) = \tmop{Sig}(\gamma|_{[t_0, t]})$,
where $\gamma |_{[t_0, t]}$ restricts the path $\gamma$ to the interval $[t_0,t]$. 
The signature can be truncated to give an approximate encoding of a path.
For example, truncating to level one approximates $\gamma$ by the chord $\gamma(t_1) - \gamma(t_0) 
\in \RR^d$. 

The tensor algebra $T((\RR^d))$ has a Lie bracket $[v, w] = v \otimes w - w \otimes v$.
We denote by $[\mathbb{R}^d, \mathbb{R}^d]$ those elements of $(\RR^d)^{\otimes 2}$ that are linear combinations of terms of the form $[v, w]$, for $v, w \in \RR^d$. These are the skew-symmetric matrices. Extending to all levels $k \geq 1$, a \tmtextit{Lie series} 
is a tensor series obtained by iteratively taking Lie
brackets and linear combinations. The space of Lie series is
$$ L ((\mathbb{R}^d)) =\mathbb{R}^d \oplus [\mathbb{R}^d, \mathbb{R}^d]
\oplus [ \RR^d, [ \RR^d, \RR^d]] + \ldots.$$

The logarithm of the signature is the \emph{log-signature}, denoted $\tmop{logSig}(\gamma)$. It is also called the {\em log-signature transform} of $\gamma$ (on $[t_0,t_1]$).
 It can be
   obtained by evaluating the power series $\log(1 + S) = \sum_{k \geq 1} \frac1k (-1)^{k-1} S^{\otimes k}$ on the signature $1 + S$. 
   It is well-known that the log-signature lies in the space of Lie series; see Chen~\cite{chen1957integration}.

   \begin{example}
       The log-signature at level two lies in $[\RR^d, \RR^d]$, as follows.
We compute
$\tmop{logSig} (\gamma)_{ij} = \tmop{Sig} (\gamma)_{ij} - \frac12 \tmop{Sig} (\gamma)_i \tmop{Sig} (\gamma)_j$.
The product rule implies $\tmop{Sig} (\gamma)_{ij} + \tmop{Sig} (\gamma)_{ji} = \tmop{Sig} (\gamma)_i \tmop{Sig} (\gamma)_j$. Hence $\tmop{logSig} (\gamma)_{ij} = -\tmop{logSig} (\gamma)_{ji} $.
   \end{example}

Like the signature, the log-signature gives an encoding of a path, which can be truncated to an approximate encoding. Log-signatures are encodings that can offer more sparsity than signatures. Moreover, they have the useful property that truncation of the log-signature preserves the property of being a Lie series. 
A \tmtextit{Lie
polynomial} is a finite Lie series. We denote the space of Lie polynomials by $L (\mathbb{R}^d) \subset L ((\mathbb{R}^d))$.
We are interested in the paths whose log-signature is finite; i.e., whose log-signature is a Lie polynomial.

\begin{example}[Straight lines] 
\label{ex:finite_segment} 
A straight line in $\RR^d$ is a path $\gamma : [t_0, t_1] \to \RR^d$ with $\gamma(t) = ta$ for some $a \in \RR^d$.
Let $\gamma_a : [0, 1] \to \RR^d$ be the straight line $\gamma_a(t) = ta$ on $[0,1]$.
Its signature is $\tmop{Sig}(\gamma_a) = \exp(a) := \sum_{k=0}^\infty \frac{1}{k!} a^{\otimes k} \in T
((\mathbb{R}^d))$. 
That is, $\tmop{logSig}(\gamma_a) = a \in \RR^d \subset L (( \RR^d ))$, a Lie polynomial of degree one.
 The straight line $\gamma(t) = ta$ on domain $[t_0, t_1]$ has signature $\tmop{Sig}(\gamma_{(t_1 - t_0)a})$, hence signatures of general straight lines are signatures of paths of the form $\gamma_a$.
Signatures are unchanged under reparametrisation and translation. Hence given a vector $b \in \mathbb{R}^d$ and increasing bijection $\tau: [0,1] \to [0,1]$, the path
$t \mapsto \gamma_a ( \tau (t )) + b 
$
has the same signature as $\gamma_a$.\end{example}

The space of paths has a concatenation product, as follows. Given
a path $\gamma : [t_0, t_1] \to \RR^d$ and a fixed $u \in [t_0, t_1]$, the path $\gamma$ divides into segments $\phi := \gamma |_{[t_0, u]}$ and $\psi := \gamma|_{[u, t_1]}$. 
The signature of $\gamma$ is  the tensor product in the tensor algebra of the signatures of the two segments: $\tmop{Sig}(\gamma) = \tmop{Sig} (\phi) \otimes \tmop{Sig} (\psi)$.
The path $\gamma$ is the concatenation of its two segments; we write this as $\gamma = \phi \star \psi$. 
The concatenation product $\phi \star \psi$ extends to general paths $\phi : [t_0, t_1] \to \RR^d$ and $\psi : [s_0, s_1] \to \RR^d$, by translating $\psi$ and its time interval
 so that $s_0 = t_1$ and $\psi(s_0) = \phi(t_1)$. 
 These transformations do not alter the signature.
In particular, the concatenation of line segments $\gamma_a \star \gamma_b$ exists and has signature $\tmop{Sig}(\gamma_a \star \gamma_b) = \exp(a) \otimes \exp(b)$.

We now describe an expression for $\tmop{logSig}(\gamma_a \star \gamma_b)$ in terms of $a$ and $b$.
The \emph{Baker--Campbell--Hausdorff (BCH) formula}~\cite{dynkin1947calculation} gives an expression for $c := \BCH(a, b) \in L((\RR^d))$ such that $\exp(c) = \exp(a) \otimes \exp(b)$. 
That is, it gives an expression for the Lie series $\tmop{logSig}(\gamma_a \star \gamma_b)$.
The BCH formula begins 
\begin{equation}
    \label{eqn:bch}
    \BCH( a, b) = a + b + \frac{1}{2} [a, b] + \frac{1}{12} \left( [a, [a,b]] + [b, [b, a]] \right) + \cdots . 
\end{equation}
A formulation in our context, together with explicit expression of the full commutator series can be found, e.g., in \cite[Section 7.3]{friz2010multidimensional}.
Note that $[a, b]$ and all higher commutators terms vanish if $a$ and $b$
are collinear.

\bigskip 
We call a rectifiable path $\gamma: [t_0,t_1] \to \RR^d$ {\it reduced} if there is no interval $[s_0,s_1]\subset[t_0,t_1]$
such that
\begin{equation}
    \label{eqn:trivial_sig}
     \tmop{Sig} ( \gamma |_{[s_0,s_1]} ) = 1.
\end{equation}
Equivalently, a path is reduced if and only if the function $t \mapsto \mathbf{S}(t) = \tmop{Sig} (\gamma |_{[t_0, t]} )$ is injective, where the equivalence holds as a consequence of the multiplicative property of the signature, $\tmop{Sig} ( \gamma |_{[s,t]}) = \mathbf{S}(s)^{-1} \otimes \mathbf{S}(t)$.
 
 \begin{remark} \label{rem} \hfill
 \begin{enumerate}[label=(\roman*)]
     \item Condition~\eqref{eqn:trivial_sig} implies that $\gamma |_{[s_0,s_1]}$ is tree-like and thus contractible to the constant path within its own image,
see  \cite{boedihardjo2016signature} and references therein.
\item Any rectifiable path has a unique reduced path with the same signature,
 up to reparamatrisation and translation, see \cite{hambly2010uniqueness} and \cite[Remark 4.1]{boedihardjo2016signature}.
    This reduced path minimizes the length (that is, the $1$-variation) among all rectifiable paths with the same signature.
    \item A sufficient condition for 
 a piecewise linear path to be 
reduced is that no two consecutive pieces are collinear. 
 \end{enumerate}
   \end{remark} 

We now state
our main result.

\begin{theorem}
\label{thm:main}
  A rectifiable reduced path 
  has polynomial log-signature if and only if it is a straight line (or a reparametrisation and translation thereof).
\end{theorem}

This paper is organized as follows. We prove the main result in Section~\ref{sec:proof}. We point out extensions to finite $p$-variation $p<2$ and failure for $p \ge 2$ in Section~\ref{sec:p-variation}. We relate our results to work of Boedihardjo et al.~\cite{boedihardjo2020path} in Section~\ref{sec:boedihardjo} and discuss similarities with a classical result of Marcinkiewicz in Section~\ref{sec:moments}.
We conclude with a polynomial perspective for piecewise linear paths in Section~\ref{sec:PL}.

\section{Proof of main result}
\label{sec:proof}

In this section, we prove Theorem~\ref{thm:main}.

\begin{proof}
Though the statement of Theorem~\ref{thm:main} pertains to rectifiable paths, our proof uses results from rough path theory, in particular that
the signature characterises a rough path up to reparametrisation.
We use this uniqueness result for rough paths that are not rectifiable.
We give definitions and precise references to rough path theory where appropriate in the proof.

Without loss of generality, paths are defined on $[0,1]$. 
The straight line $\gamma_a : [0, 1] \to \RR^d$ with $\gamma_a(t) = at$ has log-signature the finite Lie polynomial $a \in \RR^d$, see Example~\ref{ex:finite_segment}. 
The log-signature of a path $\gamma$ that is a reparametrisation or translation of $\gamma_a$ is also a finite Lie polynomial.
For the converse, take $\gamma : [0,1] \to \RR^d$ and assume that $\ell := \tmop{logSig} (\gamma)$ is a Lie polynomial with 
$N = \deg \ell < \infty$.

\underline{Step 1:} We introduce and study the path $\mathbf{X}: [0,1] \to T((\RR^d))$ with $\mathbf{X}(t) = \exp(t \ell)$. Its increments are $\mathbf{X}(s,t) := \mathbf{X}(s)^{-1} \otimes \mathbf{X}(t) = \exp ( ( t - s) \ell)$. 
Note the absence of BCH terms, since $s \ell$ and $t \ell$ are collinear.
The increments define a {\em multiplicative functional} in $T((\RR^d))$; that is, they satisfy
$\mathbf{X}(s,t) \otimes \mathbf{X}(t,u) = \mathbf{X}(s,u)$ for all $s,t,u \in [0,1]$ with $s < t < u$.
Projecting to the level $N$ truncated tensor algebra gives a path $\mathbf{X}^{\le N}$ in $T^{(N)}(\RR^d)$. Its increments $\mathbf{X}^{\le N} (s,t)$ define a multiplicative functional $T^{(N)}(\RR^d)$, by replacing $T((\RR^d))$ by its truncation. 

Denote the projection of $T((\RR^d))$ or $T^{(N)}(\RR^d)$ onto the $k$th graded component $(\RR^d)^{\otimes k}$ by $\pi_k$.
On the one hand, the increments $\mathbf{X}^{\le N}(s,t)$ are Lipschitz, since each coordinate entry is bounded 
by a constant times $|t-s|$. 
On the other hand, the projections $\pi_k ( \mathbf{X}^{\le N}(s,t)\mathbf )$ are bounded by a constant times $|t-s|^{\alpha k}$, by a defining property of (graded) H\"older regularity of such a multiplicative functional. This is compatible with the Lipschitz estimate if and only if $\alpha k \le 1$, for all $k \le N$, which leads to the best (largest) choice of the H\"older exponent being $\alpha = 1/N$.
In the language of rough path theory, $\mathbf{X}^{\le N}(s,t)$ is a multiplicative functional in $T^{(N)}(\RR^d)$ of finite $(1/N)$-H\"older regularity.
Such multiplicative functionals have a {\em unique} extension to a $T((\RR^d))$-valued 
multiplicative functional, 
by~\cite[Theorem 2.2.1]{lyons1998differential} applied with H\"older control $\omega (s,t)$ proportional to $t-s$.
Denote this unique extension by 
$\mathbf{Z}(s,t)$.
It satisfies
$\pi_k (\mathbf{Z}(s,t) ) = \pi_k ( \mathbf{X}^{\le N}(s,t) )$ for $k \le N$ and $| \pi_k ( \mathbf{Z}(s,t) )   | \le C_k |t-s|^{\alpha k}$ for all $k > N$ and some constants $C_k$. 
The increments $\mathbf{X}(s,t) = \exp ( (t-s ) \ell) \in  T((\RR^d))$ are therefore our required multiplicative functional, since they agree with $\mathbf{X}^{\le N}$ under projection and satisfy the 
required analytic condition for $\alpha = 1/N$. Thus $\mathbf{X}$ is the unique extension of $\mathbf{X}^{\le N}$.

\underline{Step 2:} We compute the signature of the path $\mathbf{X}^{\le N}$. The logarithms of $\mathbf{X}^{\le N}$ and $\mathbf{X}$ take values in the Lie polynomials (resp. Lie series). In the language of rough paths \cite{friz2010multidimensional}, the path $\mathbf{X}^{\le N}$ is then by definition a {\em weakly geometric} $(1/N)$-H\"older rough path. The path $\mathbf{X}^{\le N}$  is also, by definition, a weakly geometric 
$p$-rough path, with $p$-variation exponent $p=N$, 
 since $\alpha$-H\"older implies $1/\alpha$-variation regularity. By the definition of the signature of a weakly geometric $p$-rough path, we thus have
$$
              \tmop{Sig} (\mathbf{X}^{\le N}) = \mathbf{X}(0, 1) = \exp ( \ell ).
$$

\underline{Step 3:} We compute the signature of the path $\gamma$ in two ways. It is
$$
	     \tmop{Sig} (\gamma) = \mathbf{S}(0, 1) = \exp ( \ell ),
$$
since $\ell = \tmop{logSig} (\gamma)$. 
This signature is consistent with the notion of signature for weakly geometric rough paths, as we now explain.
The path $\gamma$ is rectifiable, by assumption, and hence has finite $1$-variation.
Assume $\gamma (0) =0$ without loss of generality.
Then $t \mapsto (1, \gamma (t)) \in \RR \oplus \RR^d \cong  T^{(1)}(\RR^d)$ is a weakly geometric $1$-rough path.
Let $\mathbf{S}$ denote the indefinite signature of $\gamma$.
The extension of $\gamma$ to a multiplicative functional, provided by \cite[Theorem 2.2.1]{lyons1998differential}  with $p=1$ and $\omega(s,t)$ proportional to the $1$-variation of $\gamma$ over $[s,t]$, equals $\mathbf{S} (s,t)$.
Hence
$$
	     \tmop{Sig} (\gamma) = \mathbf{S}(0, 1) = \exp ( \ell ).
$$

\underline{Step 4:} We use the uniqueness of the signature to show that $\gamma(t)$ is a straight line segment in $\RR^d$, with $\gamma(t) = t \ell$. 
For this, we use the space $G^*_{prc} \subset \exp ( L (( \RR^d )))$, which is defined in~\cite[Definition 2.1]{boedihardjo2016signature} through quantified estimates. For us, it suffices to note that $G^*_{prc}$ contains the signature of all weakly geometric $p$-rough paths and, in particular, all rectifiable paths, {see~\cite[Proposition 2.1]{boedihardjo2016signature}.}

A path $\gamma$ has {\em finite $p$-variation} if
$$
     \| \gamma \|_{p-\mathrm{var} } := \Big(   \sup_\pi \sum_{[s,t] \in \pi } | \gamma (t) - \gamma (s)|^p  \Big)^{1/p }< \infty,
$$
where the supremum is over all partitions of $[0,1]$.
In particular, rectifiable paths are continuous paths of finite
$1$-variation.
For every $S : [0, 1 ] \to G^*_{prc}$ with finite $p$-variation, there exists an injective path $\tilde S : [0, 1 ] \to G^*_{prc}$, unique up to reparametrisation,
such that $S(1) = \tilde S(1)$  and $S(0) = \tilde S(0)$,
by~\cite[Lemma 4.6]{boedihardjo2016signature}. 
We apply this result to $\mathbf{S}(t)$, the indefinite signature of $\gamma$. The map $\mathbf{S}$ is injective, since $\gamma$ is reduced. Thus $\mathbf{S}$ coincides, up to reparametrisation, with $\tilde{\mathbf{S}}$. 
The path $\mathbf{X}$ with $\mathbf{X}(t) = \exp ( t \ell)$ also sends $[0,1]$ injectively into $G^*_{prc}$ and satisfies $S(0) = 1$ and $S(1) =\exp ( \ell )$. By uniqueness, it must be a reparametrisation of $\mathbf{S}$. 
Reparametrization leaves $p$-variation regularity unchanged. 
The path $\mathbf{S}$ is the indefinite signature of the weakly geometric $1$-rough path $(1, \gamma)$.
The best $p$-variation exponent one has for $\mathbf{X}(t) = \exp ( t \ell)$ is $p = \deg \ell$. 
Hence $\deg \ell = 1$, or $\ell \in \RR^d$.

A $1$-rough path is determined by its projection onto the level-$1$ part of the signature; that is, its projection to $\RR^d$.
The projection of $\mathbf{X}$ to $\RR^d$ is the straight line $t \ell$. Hence $\gamma(t) = t\ell$ is a straight line in $\RR^d$. 
\end{proof} 

\begin{remark}
The proof of Theorem~\ref{thm:main} breaks if $\ell:= \tmop{logSig}(\gamma)$ is infinite: in this case the path $\mathbf{X}(t) = \exp ( t \ell) \in T (( \RR^d))$ fails to be the (indefinite) signature of any $p$-rough path since, 
as discussed in the proof, the best such $p$ is the degree of $\ell$, which here is infinite.
Hence~\cite[Lemma 4.6]{boedihardjo2016signature} does not apply to $\mathbf{X}$, and we cannot deduce that $\mathbf{S}$ and $\mathbf{X}$ agree up to reparametrisation. 
\end{remark}

Our main theorem says that any
rectifiable reduced path that is not a straight-line (or reparametrisation and translation thereof) has infinitely many non-zero terms in its log-signature. 
Consequently, the log-signature of a rectifiable path either has degree one or infinite support.
We give a second proof via  {\em smooth rough paths} \cite{bellingeri2022smooth}, which bypasses some analytic details of the first proof. Smooth rough paths are in precise analogy with  {\em smooth models} in Hairer's regularity structures \cite[Definition 6.7]{bruned2019algebraic}.

\begin{proof}[Alternative proof of Theorem~\ref{thm:main}] 
We assume for simplicity that $\gamma$ is smooth, but the argumment is identical for continuously differentiable $\gamma$ and, more generally, absolutely continuous $\gamma$, writing $t$-almost surely 
when doing calculus. Starting with a general rectifiable path $\gamma$, reparametrisation by running length does not alter the signature and yields an absolutely continuous paths, so there is no loss of generality.

Let $N$ be the degree of Lie polynomial $\ell$.
Let $\mathbf{X} (t)$ denote the projection of $\exp(t \ell)$ to $T^{(N)} ( \RR^d )$.
It is a level-$N$ smooth geometric rough path in the sense of \cite{bellingeri2022smooth}. 
By the fundamental theorem of smooth geometric rough paths~\cite[Theorem 2.8]{bellingeri2022smooth}, it has a unique extension $\mathbf{Z}(t) \in T (( \RR^d ))$ that satisfies the minimality condition that $\mathbf{Z}^{-1} (t) \otimes \dot{\mathbf{Z}} (t)$ lies in the truncated space $T^{(N)} ( \RR^d )$ rather than $T (( \RR^d ))$. The extension therefore satisfies
$\mathbf{Z}^{-1} (t) \otimes \dot{\mathbf{Z}} (t)  = \mathbf{X}^{-1} (t) \otimes \dot{\mathbf{X}} (t) \in T^{(N)} ( \RR^d )$.
In our case the right hand side is $\ell \in T^{(N)} ( \RR^d )$ and solving the ODE in $T (( \RR^d ))$ with $\mathbf{Z}(0)=\mathbf{X}(0)=1$, gives $\mathbf{Z}(t) = \exp(t \ell)) \in T^{(N)} ( \RR^d )$. Consistency of $\mathbf{Z}$ with the Lyons' extension is verified in~\cite[Proposition 2.14]{bellingeri2022smooth}. 
This is an essentially computation-free alternative to Step 1 from our previous proof.

For Step 2, we view $\gamma$  as a level-$1$ smooth geometric rough path, whose unique minimal extension is $\mathbf{S}(t)$, the indefinite signature of $\gamma$. Step 3 is as before. 
Step 4 is also as before, and again relies on \cite[Lemma 4.6]{boedihardjo2016signature}. (Specialising~\cite[Lemma 4.6]{boedihardjo2016signature} to the setting of smooth geometric rough paths does not seem to lead to major simplifications.) As in our previous proof, it follows that $\mathbf{X}$ and $\mathbf{S}$ are reparametrizations of each other and, since Riemann-Stieltjes integration is not affected 
by such reparametrizations, we have, using $d \gamma = \mathbf{S}^{-1} \otimes d \mathbf{S}$,
\[ \ell = \int_0^1 \mathbf{X}(t)^{- 1} d \mathbf{X}(t) = \int_0^1
   \mathbf{S}(t)^{- 1} d \mathbf{S}(t) = \int_0^1 d \gamma = \gamma (1) -
   \gamma (0) \in \mathbb{R}^d . \]
In particular, $\ell \in \mathbb{R}^d$, and we conclude as before.
\end{proof}

\section{Beyond rectifiability}
\label{sec:p-variation}

Theorem~\ref{thm:main} generalises from rectifiable paths to continuous paths of finite $p$-variation, provided $p<2$. The signature transform remains well-defined by iterated Young integration, see e.g. \cite{lyons2002system, friz2010multidimensional}.
We have the following generalization of Theorem~\ref{thm:main} with the same proof. 

\begin{theorem}
\label{thm:main_rough}  
Let $p < 2$. A continuous reduced path of finite $p$-variation has polynomial log-signature if and only if it is a straight line (or a reparametrisation and translation thereof).
\end{theorem}

Theorem~\ref{thm:main_rough} is false if $p
\geqslant 2$. 
Indeed, it suffices to consider $t \mapsto \exp (t \ell) \in T^{(2)} (\mathbb{R}^d)$ 
for $\ell$ a non-zero homogenous Lie
polynomial of degree $m = 2$. Such rough paths are  called {\em pure-$m$ rough paths}, or pure area rough paths when $m=2$.

\section{Proof via BGS estimates}
\label{sec:boedihardjo}

Another route to Theorem~\ref{thm:main} was suggested to us by H. Boedihardjo,
in terms of a careful application of the analytic estimates of~\cite[Theorem 2.2]{boedihardjo2020path}. 
This replaces the use of~\cite[Lemma 4.6]{boedihardjo2016signature} that was central to our two earlier proofs.
All three proofs use results from rough path theory. It is an open problem to prove Theorem~\ref{thm:main} via an approach that studies only rectifiable paths.

\begin{proof}[Yet another proof of Theorem~\ref{thm:main}] Given a tensor series ${\bf T} \in T ((\RR^d))$, we define
\begin{equation} \label{eq:Lp}
L_{p}({\bf T}):= \limsup_{k\rightarrow\infty}\left(\left(\frac{k}{p}\right)!\left\Vert \pi_k ( \bf T) \right\Vert \right)^{p/k},
\end{equation}
where $\| \cdot \|$ denotes a norm of order $k$ tensors, for any $k$, that is compatible in the sense that $\| a \otimes b \| \le \| a \| \, \| b \|$. Applied to the signature of $\gamma$, we see that 
\begin{equation} \label{eq:LpSig}
L_1(\tmop{Sig}(\gamma)) \le \| \gamma \|_{1\text{-var}} \quad \text{and} \quad L_p(\tmop{Sig}(\gamma)) = 0 \text{ for } p > 1,\end{equation}
as follows. This is a consequence of 
\[ \| \pi_k (\tmop{Sig} (\gamma)) \| = \frac{1}{k!} \left\| \int_{[0, 1]^k} \! \! \! \!
   \dot{\gamma} \otimes \cdots \otimes \dot{\gamma} d (t_1, \ldots, t_k) \right\|
   \leqslant \frac{1}{k!} \int_{[0, 1]^k}  \! \! \! \| \dot{\gamma} \otimes \cdots \otimes
   \dot{\gamma} \| d (t_1, \ldots, t_k) \]
\[ \leqslant \frac{1}{k!} \int_{[0, 1]^k}  \| \dot{\gamma} \| \ldots \|  \dot{\gamma}
   \| d (t_1, \ldots, t_k) = \frac{1}{k!} \left( \int_0^1 \| \dot{\gamma} \|
   \tmop{dt} \right)^k = \frac{1}{k!} \| X \|_{1 - \tmop{var}}^k. \]
Here $d \gamma = \dot{\gamma} dt$, which entails no loss of generality. Reparametrisation by running length does not alter the signature and yields an absolutely continuous path. Alternativley, use facts on Stieltjes integrals that directly justify the above estimation.

The result~\cite[Theorem 2.2]{boedihardjo2020path} states that (under the additional assumption that the tensor norms are projective, which entails no loss of generality in our finite-dimensional setup, with base space $\RR^d$)
for each
$m\geqslant1$, there exists a constant $c(m,d)\in(0,1]$ 
such that 
\[
c(m,d)\|\pi_{m}(\ell )\|\leqslant L_{m}( \tmop{Sig} ({\bf X}) )\leqslant  \|\pi_{m}(\ell)\|
\]
for every pure-$m$ rough path ${\bf X}(t)=\exp(t \ell) \in T^{(m)} (\RR^d)$, with $\deg \ell = m$.
But by Step 1 and 2 in our previous proofs, or by \cite[Proposition 2.2]{boedihardjo2020path},  we know $ \tmop{Sig} ({\bf X})$
is $\exp ( \ell) \in T (( \RR^d ))$. This implies the estimate
\begin{equation} \label{eq:cmd}
c(m,d)\|\pi_{m}(\ell )\|\leqslant L_{m}( \exp ( \ell )  )\leqslant  \|\pi_{m}(\ell)\| .
\end{equation}
If $\tmop{Sig}(\gamma) = \exp ( \ell)$,  for some Lie polynomial $\ell$ of degree $m$, then 
\eqref{eq:LpSig} and \eqref{eq:cmd} implies $m=1$. It follows that $\ell \in \RR^d$ and we can conclude as before.
\end{proof}

\section{Connections to Marcinkiewicz's theorem}
\label{sec:moments}

Theorem~\ref{thm:main} has the following implication.

\begin{corollary}
\label{cor:poly_Pm}
  Let $P$ be a Lie polynomial of degree $m \geqslant 2$. 
  Then $P$ cannot be the log-signature transform of a rectifiable path.
\end{corollary}

We compare this to a classical result of  Marcinkiewicz~\cite{marcinkiewicz1939propriete}, see also \cite[Theorem B]{lukacs1958some}, 
which we restate in a form that exhibits their similarity. Let $X=X(\omega)$ be a random variable with values in $\mathbb{R}^d$, with finite moments of all orders.
Its {\em moment transform} is
\[ 
\mu (X) 
\assign \mathbb{E} (\exp (X)) = 1 + \sum_{k \geqslant 1} \frac{1}{k!} \mathbb{E} (X^{\otimes k}).
   \]
The moment transform lies
the symmetric algebra over $\mathbb{R}^d$, 
denoted $\tmop{Sym}
((\mathbb{R}^d))$,
which consists of series of symmetric tensors. This space is isomorphic to the space of power series in $d$ commuting indeterminates.
 The {\em log-moment} (or {\em cumulant}) transform is
 $$
    \kappa (X) := \log  \mu(X) \in \tmop{Sym} ((\mathbb{R}^d)).
 $$
 The entries of 
 $\mu(X)$ and $\kappa(X)$ 
are the multivariate moments and cumulants of $X$, 
up to factorial constants.

\begin{example}
    If $X \sim N(b,a)$ is normally distributed with mean $b \in
\mathbb{R}^d$ and covariance $a \in \mathbb{R}^d \otimes \mathbb{R}^d$, the log-moment transform is $\kappa(X) = b + a/2$, a degree two polynomial in $\tmop{Sym}((\mathbb{R}^d))$.
\end{example}

\begin{theorem}[\cite{marcinkiewicz1939propriete}]
\label{thm:marc}
  Let $P$ be a polynomial of degree $m \geqslant 3$. Then 
  $P$ cannot be the log-moment transform of a probability measure (with all moments finite). 
\end{theorem}

We discuss
steps towards a mutual generalization of Corollary~\ref{cor:poly_Pm} and Theorem~\ref{thm:marc}. 
Let $p \in [1, \infty)$. Given a random weakly geometric $p$-rough path $\mathbf{X} = \mathbf{X} (\omega)$ in $\RR^d$, its signature and log-signature are $T((\RR^d))$-valued random variables. Assuming (componentwise) integrability, we 
define
the {\em expected signature} and {\em signature cumulant} by
$$
       \boldsymbol{\mu} (\mathbf{X}) := \mathbb{E} ( \tmop{Sig} (\mathbf{X})), \quad \boldsymbol{\kappa} 
       (\mathbf{X}) := \log 
       \mathbb{E} ( \tmop{Sig} (\mathbf{X}))
       \in   T((\RR^d)).
 $$

The signature cumulant of the Brownian rough path is finite, as follows.

 \begin{example}
 \label{ex:brownian}
  Let $X\sim \mathrm{Bm}(a,b)$, meaning that $X=(X(t):0 \le t \le 1)$, is a 
Brownian motion with drift $b$ and covariance $a$. That is, $Y(t) := X(t) - bt$ defines a centered Gaussian process with covariance $\mathbb{E} ( Y(s) \otimes Y(t)) = a \min \{ s, t \}$; in particular $X (t)  \sim N( b t, a t)$.
Let $\mathbf{X} \sim \mathrm{Brp}(b,a)$ by which we mean that $\mathbf{X}$ is the {\em Brownian rough path} obtained from $X$ by iterated Stratonovich integration. Then  $ \boldsymbol{\mu} (\mathbf{X}) = \exp (  b + a/2)$. 
 See e.g. \cite[Chapter 3]{friz2020course},~\cite[Section 7.1]{amendola2019varieties} for an algebraic geometry perspective, and \cite{friz2022unified} for a far-reaching extension to a general semimartingale.
Equivalently,
$$
      \boldsymbol{\kappa} (\mathbf{X}) := b + a/2 , 
$$
which is a degree two polynomial in $T((\mathbb{R}^d))$.    
 \end{example}

The naive guess is that the signature cumulant of any Gaussian rough path (e.g. in the sense of \cite[Chapter 15]{friz2010multidimensional}) is finite. This is wrong, by the following example.

\begin{example}
    Take independent $\mathbf{X}^i \sim \mathrm{Brp}(0,a_i)$, for $i=1,2$. 
Write $\mathbf{X}^1 \star \mathbf{X}^2$ for the concatenation of these random rough paths. The expected signature is $\exp (a_1 /2 ) \otimes \exp (a_2 /2 )$. By the BCH formula, $$\boldsymbol{\kappa} ( \mathbf{X}^1 \star \mathbf{X}^2 ) = \mathrm{BCH} (a_1/2, a_2/2) = a_1/2 + a_2/2 + [a_1,a_2]/8 + ...\in T(( \RR^d)). $$
Unless $a_1$ and $a_2$ are collinear, this is non-polynomial, cf. Corollary~\ref{cor:2} below. 
\end{example}

The above example does not have stationarity of increments, in contrast to Example~\ref{ex:brownian}. See e.g. \cite[ Chapter 13]{friz2010multidimensional} or \cite{friz2017general}.

\begin{definition}[Brownian-like signatures] 
Given a random weakly geometric $p$-rough path $\mathbf{X} = \mathbf{X} (\omega)$ we call its signature $\tmop{Sig} (\mathbf{X})$ {\em Brownian like} if its signature cumulants $ \boldsymbol{\kappa} (\mathbf{X})$ are well-defined and polynomial, that is, finite in $T((\mathbb{R}^d))$.
\end{definition} 

For example, \cite{friz2015physical} computed the expected signature of the so-called magnetic Brownian (rough) path to be of the form $\exp ( \tilde b + a/2)$ where $\tilde b = b^1 + b^2 \in L^2 (\RR^d)$ and $a  \in (\RR^d)^{\otimes 2}$ is a symmetric matrix.
Note that $\tilde b + a/2$ is a  polynomial of degree two in  $T ((\RR^d ))$.
We also note the example of the deterministic weakly geometric $N$-rough path, given by $t \mapsto \exp (t  \ell) \in T^{(N)} (\RR^d )$, for $\ell \in L(\RR^d)$ of degree $N$, which has 
$\boldsymbol{\kappa} = \tmop{logSig}$ equal to the Lie polynomial $\ell$. We suspect that Brownian like signature are related to signatures of ``higher order'' Brownian rough paths. This is left to future work.

\section{Piecewise linear paths}
\label{sec:PL}

We specialise to piecewise linear paths, studying connections to systems of polynomial equations. The piecewise linear path with $m$ pieces $v_1, \ldots, v_m \in \RR^d$ is the concatenation $\gamma := \gamma_{v_1} \star \gamma_{v_2} \star  \cdots \star \gamma_{v_m}$.

\begin{corollary}
    \label{cor:2a}
    Let $\gamma$ be a piecewise linear path with $m$ pieces and no two consecutive pieces collinear. The log-signature of $\gamma$ is finite if and only if $m = 1$.
\end{corollary}

\begin{proof}
A sufficient for a piecewise linear path to be reduced is that no two consecutive pieces are collinear.
Hence $\gamma$ is a straight line, by Theorem~\ref{thm:main}. Since consecutive pieces are not collinear, the straight line has $m=1$ piece.
\end{proof}

We rephrase Corollary~\ref{cor:2a} as a statement about the Baker Campbell Hausdorff formula. Given $v_1, \ldots, v_m \in \RR^d$, we define the iterated BCH formula to be the Lie series $c := \BCH(v_1, \ldots, v_m) \in L ((\mathbb{R}^d))$ such that $\exp(c) = \exp(v_1) \otimes \cdots \otimes \exp(v_m)$. When $m=1$, we set $\BCH(v_1) = v_1$. For $m=2$, see~\eqref{eqn:bch}. We suspect that the following may be known to experts in the Lie algebra community. 

\begin{corollary}
\label{cor:2} 
Fix $v_1, \ldots, v_m \in \RR^d$ with no consecutive $v_i$ collinear. The iterated Baker Campbell Haussdorff formula $\BCH(v_1, \ldots, v_m) \in L((\RR^d))$ has finitely many non-zero terms if and only if $m = 1$. 
\end{corollary}

\begin{proof}
The log-signature of the piecewise linear path with pieces $v_1, \ldots, v_m$ is 
$\BCH(v_1, \ldots, v_m)$, which is finite if and only if $m=1$, by Corollary~\ref{cor:2a}.
\end{proof}

The task of recovering a path 
from its truncated signature was studied via solving systems of polynomial equations in~\cite{amendola2019varieties, pfeffer2019learning}. 
We apply this perspective to Corollaries~\ref{cor:2a} and~\ref{cor:2}. 
Consider the piecewise linear path with pieces $v_1, \ldots, v_m \in \RR^d$.
Its log-signature at level $k$ is a tensor in $(\RR^d)^{\otimes k}$ that is a rational linear combination of terms $v_{i_1} \otimes \cdots \otimes v_{i_k}$ for some $i_1, \ldots, i_k \in \{1, \ldots, m\}$.

\begin{example}
The log-signature at level two is \begin{equation}
    \label{eqn:level2}
\frac12 \sum_{1 \leq i < j \leq m} ( v_i \otimes v_j - v_j \otimes v_i ).
\end{equation}
The log-signature at level three is
\begin{equation}
    \label{eqn:level3}
 \frac{1}{12} \sum_{i \neq j } (v_i^{\otimes 2} \otimes v_j + v_i \otimes v_j^{\otimes 2} ) + \frac13 \! \! \sum_{(i,j,k) \in S_1} \! \! \! \! \! \! v_i \otimes v_j \otimes v_k 
 - \frac16 \! \! \sum_{(i,j,k) \in S_2} \! \! \! \! \! \! v_i \otimes v_j \otimes v_k ,
\end{equation}
 where $S_1 = \{ (i,j,k) : i \! < \! j \! < \! k \text{ or } i \! >\! j \! > k \}$, $S_2 = \{ (i,j,k) : i \! < \! j \! > \! k \text{ or } i \! > \! j \! < \! k \}$.
 \end{example}

Each entry of the level $k$ log-signature is a degree $k$ polynomial in the $md$ entries $v_{ij}$ of the $m$ pieces $v_i \in \mathbb{R}^d$.
Hence, for piecewise linear paths, Theorem~\ref{thm:main} describes the real solutions to an (infinite) system of polynomial equations. 
In fact, the vanishing of finitely many levels of the log-signature suffice to conclude that the path is a straight line. 

\begin{corollary}
\label{cor:n2}
Given $n_1 \geq 1$ and $m \geq 2$, there exists a smallest integer $n_2 = n_2 ( n_1, m) \geq n_1$ such that 
there is no piecewise linear path with $m$ pieces, and no consecutive pieces collinear, whose log-signature vanishes at levels $n_1, n_1 + 1, \ldots, n_2$.
\end{corollary}

\begin{proof}
Fix a piecewise linear path with pieces $v_1, \ldots, v_m \in \RR^d$. Each entry of the log-signature is a polynomial in $\mathbb{Q}[ v_{ij} : 1 \leq i \leq m, 1 \leq j \leq d]$.
Let $I_{n_1, k}$ be the ideal generated by the entries of log-signature at levels $n_1, n_1 + 1, \ldots, k$, for $k \geq n_1$.
As $k$ increases, we obtain $I_{n_1, k} \subseteq I_{n_1, k+1} \subseteq \cdots$, which stabilizes at some ideal $I_{n_1, n_2}$, by Noetherianity, see e.g.~\cite[Theorem 7 of \S5]{cox1994ideals}. The log-signature vanishes at level $n_1$ and above whenever the pieces lies in the vanishing locus of $I_{n_1, n_2}$. Hence the vanishing locus of $I_{n_1, n_2}$ contains no real $v_i$ with no consecutive pieces collinear, by Corollary~\ref{cor:2a}, since $m \geq 2$. 

We rule out dependence of $n_2$ on $d$. Assume there is a path in $\RR^d$ with $m$ pieces, no two consecutive pieces collinear, whose log-signature at levels $n_1, n_1 + 1, \ldots, n_2$ vanishes. Embedding the path into $\RR^{d'}$ for $d' > d$ gives a path with this property in $\RR^{d'}$. For $d' < d$, there exist projections of the path onto $\RR^{d'}$ in which consecutive increments remain non-collinear. The log-signature of the projection also vanishes at levels $n_1, n_1 + 1, \ldots, n_2$. Hence $n_2$ is a function of just $n_1$ and $m$.
\end{proof}

Corollary~\ref{cor:n2} suggests computing the required upper level $n_2$ for different starting levels $n_1$ and numbers of pieces $m$.  We leave the study of the function $n_2 = n_2(n_1, m)$ to future work and conclude this article with two examples.

\begin{example}[Two pieces]
Let $n_1 = 2$. The level two log-signature is $\frac12 [ v_1, v_2]$, which vanishes if and only if $v_1$ and $v_2$ are collinear. Hence $n_2(2,2) = 2$: any path with two pieces whose log-signature vanishes at level two is a straight line. Now we consider $n_1 = 3$. The level three log-signature is
 $$ [v_1, [v_1,v_2]] + [v_2, [v_2,v_1]] ,$$
which vanishes if and only if $v_1, v_2$ are collinear. Hence $n_2(3,2) = 3$: no piecewise linear path with two non-collinear pieces has vanishing level three log-signature.
\end{example}
 
\begin{example}[Three pieces]
Let $n_1 = 2$. 
There exist paths with three pieces, no consecutive pieces collinear, whose level two signature vanishes. For example, let
$$ v_1 = \begin{bmatrix} 1 \\ 1 \end{bmatrix} , \quad v_2 = \begin{bmatrix} 1 \\ -1 \end{bmatrix} , \quad v_3 = \begin{bmatrix} a \\ 1 \end{bmatrix}, $$
for any $a \in \mathbb{R}$. Hence $n_2(2,3) \geq 3$.
Setting~\eqref{eqn:level2} and~\eqref{eqn:level3} to zero gives the ideal $I_{2,3}$ from the proof of Corollary~\ref{cor:n2}. 
A Macualay2~\cite{M2} computation shows that $I_{2,3}$
has three components: $v_1 + v_2 = 0$, $v_3 + v_2 = 0$ and $v_1, v_2, v_3$ all parallel. 
In all three components, there exist two consecutive pieces that are collinear. In the first two components, the path is not reduced and the corresponding reduced path is a straight line. In the third component, the path is a straight line. Hence all components lie in $I_{2,k}$ for all $k \geq 2$, and therefore $n_2(2,3) = 3$. Next we consider $n_1 = 3$. Setting~\eqref{eqn:level3} to zero reveals that the third order log-signature vanishes for paths with pieces $v_1, v_2, v_3$ that satisfy $v_1 + 3 v_2 + v_3 = 0$. Hence $n_2(3,3) \geq 4$.
\end{example}

\section*{Acknowledgements}

PFK acknowledges support from the German science foundation (DFG) via the cluster of excellence MATH+, via a MATH+ Distinguished Fellowship, and DFG Research Unit FOR2402 `Rough paths, stochastic partial differential equations and related topics'. He thanks Horatio Boedihardjo for helpful discussions related to this paper.
TL was funded in part by the EPSRC [grant number EP/S026347/1], in part by The Alan Turing Institute under the EPSRC grant EP/N510129/1, the Data Centric Engineering Programme (under the Lloyd’s Register Foundation grant G0095), the Defence and Security Programme (funded by the UK Government) and the Office for National Statistics \& The Alan Turing Institute (strategic partnership) and in part by the Hong Kong Innovation and Technology Commission (InnoHK Project CIMDA).
AS was supported by the Society of Fellows at Harvard University.

\bibliographystyle{alpha}
 \bibliography{references}

\end{document}